\documentclass[11pt]{amsproc}
\usepackage{amssymb}
\usepackage{amsfonts}

\theoremstyle{plain}
\newtheorem{acknowledgement}{Acknowledgement}

\newtheorem{conjecture}{Conjecture}

\newtheorem{lemma}{Lemma}

\newtheorem{theorem}{Theorem}
\numberwithin{equation}{section}

\begin{document}
\title[Random runners are very lonely]{Random runners are very lonely}
\author{Sebastian Czerwi\'{n}ski}
\address{Faculty of Mathematics, Computer Science, and Econometrics,
University of Zielona G\'{o}ra, 65-546 Zielona G\'{o}ra, Poland}
\email{s.czerwinski@wmie.uz.zgora.pl}

\begin{abstract}
Suppose that $k$ runners having different constant speeds run laps on a
circular track of unit length. The Lonely Runner Conjecture states that,
sooner or later, any given runner will be at distance at least $1/k$ from
all the other runners. We prove that, with probability tending to one, a much
stronger statement holds for random sets in which the bound $1/k$ is
replaced by \thinspace $1/2-\varepsilon $. The proof uses Fourier analytic
methods. We also point out some consequences of our result for colouring of
random integer distance graphs.
\end{abstract}

\maketitle

\section{Introduction}

Suppose that $k$ runners run laps on a unit-length circular track. They all
start together from the same point and run in the same direction with
pairwise different constant speeds $d_{1},d_{2},\ldots ,d_{k}$. At a given
time $t$, a runner is said to be \emph{lonely} if no other runner is within a distance of $1/k$, both in front and rear. The Lonely Runner Conjecture states that for every runner there is a time at which he is lonely. For instance if $k=2$, one can imagine easily that at some time or other, the two runners will find themselves on antipodal points of the circle, both becoming lonely at that moment.

To give a precise statement, let $\mathbb{T}=[0,1)$ denote the \emph{circle}
(the one-dimensional torus). For a real number $x$, let $\{x\}$ be the
fractional part of $x$ (the position of $x$ on the circle), and let $%
\left\Vert x\right\Vert $ denote the distance of $x$ to the nearest integer
(the circular distance from $\{x\}$ to zero). Notice that $\left\Vert
x-y\right\Vert $ is just the length of the shortest circular arc determined
by the points $\{x\}$ and $\{y\}$ on the circle. It is not difficult to see that
the following statement is equivalent to the Lonely Runner Conjecture.

\begin{conjecture}
\label{Wills}For every integer $k\geqslant 1$ and for every set of positive
integers $\{d_{1},d_{2},\ldots ,d_{k}\}$ there exists a real number $t$ such
that 
\begin{equation*}
\left\Vert td_{i}\right\Vert \geqslant \frac{1}{k+1}
\end{equation*}%
for all $i=1,2,\ldots ,k$.
\end{conjecture}

The above bound is sharp as is seen for the sets $\{1,2,\ldots ,k\}$. The paper of Goddyn and Wong \cite{GoddynWong} contains items of  interesting exemplars of such extremal sets. The problem was posed for the first time by \mbox{Wills \cite{Wills}} in connection to Diophantine approximation. Cusick \cite{Cusick} raised the same question independently, as a view obstruction problem in Discrete Geometry (cf. \cite%
{BrassMoserPach}). Together with Pomerance \cite{CusicPomerance}, he
confirmed the validity of the conjecture for $k\leqslant 4$. Bienia et al. \cite%
{Bienia} gave a simpler proof for $k=4$ and found interesting application to
flows in graphs and matroids. Next the conjecture was proved for $k=5$ by
Bohman et al. \cite{BohmanHolzmanKleitman}. A simpler proof for that case
was provided by Renault \cite{Renault}. Recently the case $k=6$ was
established by Barajas and Serra \cite{BarajasSerra}, using a new promising
idea.

Let $D=\{d_{1},d_{2},\ldots ,d_{k}\}$ be a set of $k$ positive integers.
Consider the quantity 

\begin{equation*}
\kappa (D)=\sup_{x\in \mathbb{T}}\min_{d_{i}\in D}\left\Vert
xd_{i}\right\Vert
\end{equation*}%
and the related function $\kappa (k)=\inf \kappa (D)$, where the infimum is
taken over all $k$-element sets of positive integers. So, the Lonely Runner
Conjecture states that $\kappa (k)\geqslant \frac{1}{k+1}$. The trivial
bound is $\kappa (k)\geqslant \frac{1}{2k}$, as the sets $\{x\in \mathbb{T}%
:\left\Vert xd_{i}\right\Vert <\frac{1}{2k}\}$ simply cannot cover the whole
circle (since each of them is a union of $d_{i}$ open arcs of length $\frac{1%
}{kd_{i}}$ each). Surprisingly, nothing much better was proved so far.
Currently the best general bound is 

\begin{equation*}
\kappa (k)\geqslant \frac{1}{2k-1+\frac{1}{2k-3}}
\end{equation*}%
for every $k\geqslant 5$ \cite{Chen}. A slightly improved inequality $\kappa
(k)\geqslant \frac{1}{2k-3}$ holds when $k\geqslant 4$ and $2k-3$ is prime 
\cite{ChenCusic}. Using the probabilistic argument we proved in \cite%
{CzerwinGrytczuk} that every set $D$ contains an element $d$ such that 

\begin{equation*}
\kappa (D\setminus \{d\})\geqslant \frac{1}{k}.
\end{equation*}%
%T
In this paper we prove another general result supporting the Lonely Runner
Conjecture.

\begin{theorem}
Let $k$ be a fixed positive integer and let $\varepsilon >0$ be fixed real number. Let $D\subseteq \{1,2,\ldots ,n\}$ be a $k$-element subset chosen
uniformly at random. Then the probability that $\kappa (D)\geqslant \frac{1}{%
2}-\varepsilon $ tends to $1$ with $n\rightarrow \infty $.
\end{theorem}

The proof uses elementary Fourier analytic technique for subsets of $\mathbb{%
Z}_{p}$. We give it in the next section. In the last section we point to a
striking consequence of our result for colouring of integer distance graphs. 

\section{Proof of the main result}

Let $k$ be a fixed positive integer and let $p\geqslant k$ be a prime
number. For $a\in \mathbb{Z}_{p}$, let $\left\Vert a\right\Vert _{p}=\min
\{a,p-a\}$ be the circular distance from $a$ to zero in $\mathbb{Z}_{p}$.
We will need the following notion introduced by \cite{Ruzsa}. Let $L$ be a
fixed positive integer. A set $D=\{d_{1},\cdots,d_{k}\}\subseteq \mathbb{Z}_{p}$
is called $L$\emph{-independent} in $\mathbb{Z}_{p}$ if equation 

\begin{equation*}
d_{1}x_{1}+d_{2}x_{2}+\ldots +d_{k}x_{k}=0
\end{equation*}%
has no solutions satisfying 

\begin{equation*}
0<\sum\limits_{i=1}^{k}\left\Vert x_{i}\right\Vert _{p}\leqslant L.
\end{equation*}%

We will show that for appropriately chosen $L$, any $L$-independent set can
be pushed away arbitrarily far from zero. Then we will demonstrate that for
such $L$, almost every set in $\mathbb{Z}_{p}$ is $L$-independent.

Let $f:\mathbb{Z}_{p}\rightarrow \mathbb{C}$ be any function and let $\hat{f}%
:\mathbb{Z}_{p}\rightarrow \mathbb{C}$ denote its Fourier transform, that is 

\begin{equation*}
\hat{f}(r)=\sum_{x\in \mathbb{Z}_{p}}f(x)\omega ^{rx},
\end{equation*}%
where $\omega =e^{\frac{2\pi }{p}i}$. For a set $A\subseteq \mathbb{Z}_{p}$,
by $A(x)$ we denote its characteristic function. We will make use of the
following basic properties of the Fourier transform:

\begin{description}
\item[(F1)] $\left\vert \hat{f}(r)\right\vert =\left\vert \hat{f}%
(-r)\right\vert $ for every $r\in \mathbb{Z}_{p}$.

\item[(F2)] $f(x)=\frac{1}{p}\sum_{r\in \mathbb{Z}_{p}}\hat{f}(r)\omega
^{-rx}$ for every $x\in \mathbb{Z}_{p}$.

\item[(F3)] $\hat{A}(0)=\left\vert A\right\vert $ for every subset of $%
\mathbb{Z}_{p}$.
\end{description}

In the lemma below we give a bound for the Fourier coefficient $\hat{A}(r)$
for the sets of the form 

\begin{equation}
A=\left\{ s,s+1,\ldots ,l\right\} ,
\tag{*}
\end{equation}%
where $l$ and $s$ are elements of $\mathbb{Z}_{p}$, such that $s<l$. This
bound does not depend on $l$ and $s$.
The following lemma can be easily proved, as for instance in \cite{Book}(p. 39). We proved this for a reader convenience. 
\begin{lemma}\label{lem}
\label{wsp Fouriera}If $0<r<\frac{p}{2}$, then 
\begin{equation*}
\left\vert \hat{A}(r)\right\vert \leqslant \frac{p}{2r}.
\end{equation*}
\end{lemma}

\begin{proof}
By simple calculations we have
\begin{align*}
|\hat{A}(r)|=\Big|\sum_{x=s}^{l}\omega^{rx}\Big|=%
\Big|\frac{\omega^{r(l+1)}-\omega^{rs}}{\omega^{r}-1}\Big|=&\\
\Big|\frac{\omega^{\frac{r(l+s+1)}{2}}}{\omega^{\frac{r}{2}}}\cdot
\frac{\omega^{\frac{r(l+1-s)}{2}}-\omega^{\frac{-r(l+1-s)}{2}}} {\omega^{\frac{r}{2}}-\omega^{\frac{-r}{2}}}\Big|=&
\Big|\frac{\sin(\frac{\pi r}{p})} {\sin(\frac{\pi r}{p})}\Big|.
\end{align*}%

Using inequality $\sin (x)\geqslant \frac{2x}{\pi }$ for $x<\frac{\pi }{2}$,
we get 

\begin{equation*}
\left\vert \hat{A}(r)\right\vert \leqslant \frac{p}{2r}.\qedhere
\end{equation*}

\end{proof}

Now, we state and prove the aforementioned property of $L$-independent sets.

\begin{theorem}
\label{tw L-niezal}Let $0<\varepsilon <\frac{1}{2}$ be a fixed real number.
Let $D$ be a $k$-element, $L$-independent set in $\mathbb{Z}_{p}$, where%
\begin{equation*}
L> \sqrt{\frac{k^33^{k-1}}{2^{k+1}\varepsilon ^{2k}}}.
\end{equation*}%
Then%
\begin{equation*}
\kappa (D)\geqslant 1/2-\varepsilon .
\end{equation*}
\end{theorem}

\begin{proof}
Let \[C=\{x\in\mathbb{Z}_p:(\frac{1}{4}-\frac{\varepsilon}{2})p<x<(\frac{1}{4}+\frac{\varepsilon}{2})p\} \] and let $C(x)$ be the characteristic function of the set $C$.  
Define convolution of two functions $f$ and $g$ by 
\[(f*h)(x)=\sum_{y\in\mathbb{Z}_p }f(y)\cdot g(x-y).\] 
Denote by $B(x)=(C*C)(x)$ convolution of function $C$ with itself.
It is easy to see that $\hat{B}(r)=\hat{C}(r)\cdot \hat{C}(r)$ for all $r\in\mathbb{Z}_p$.

So, if we find $t\in \mathbb{Z}_{p}$ such that $tD\subseteq \mathrm{supp} B$, 
where $\mathrm{supp} B=\{x\in\mathbb{Z}_p:B(x)\ne 0\}$, then at the same time we push the set $D$ away into the small arc $\left( \frac{1}{2}%
-\varepsilon ,\frac{1}{2}+\varepsilon \right) $ on the torus $\mathbb{T}$.

Then the expression 

\begin{equation*}
I=\sum_{t\in \mathbb{Z}_{p}}B(td_{1})B(td_{2})\cdots B(td_{k})
\end{equation*}%
counts those numbers $t$ which push the set $D$ away to a distance $\frac{1}{2%
}-\varepsilon $ from zero. We will show that $I\neq 0$. From properties of  the Fourier transform results that

\begin{equation*}
I=\sum_{t\in \mathbb{Z}_{p}}\left( \frac{1}{p}\sum_{r_{1}\in \mathbb{Z}_{p}}%
\hat{B}(r_{1})\omega ^{-td_{1}r_{1}}\right) \cdots \left( \frac{1}{p}%
\sum_{r_{k}\in \mathbb{Z}_{p}}\hat{B}(r_{k})\omega ^{-td_{k}r_{k}}\right) .
\end{equation*}%

Denoting  $\overset{\rightarrow}{r}=(r_1,r_2,\cdots,r_{k})$, we get
\begin{equation*}
p^{k}I=\sum_{\overset{\rightarrow}{r}\in \mathbb{Z}_{p}^k}%
\hat{B}(r_{1})\cdots \hat{B}(r_{k})\sum_{t\in \mathbb{Z}_{p}}\omega
^{-t(d_{1}r_{1}+\cdots+d_{k}t_{k})}.
\end{equation*}%

The expression $\sum_{t}\omega ^{-t(d_{1}r_{1}+\cdots+d_{k}t_{k})}$ is
equal to $p$ when
\begin{equation}
d_{1}r_{1}+\cdots+d_{k}r_{k}\equiv 0 \pmod p,  \tag{**}
\end{equation}
and is equal to zero in the contrary case.  As a consequence we may write 
\begin{equation*}
p^{k-1}I=\sum_{\overset{\rightarrow}{r}\in \mathbb{Z}_{p}^k}%
\hat{B}(r_{1})\cdots \hat{B}(r_{k})R(\overset{\rightarrow}{r}),
\end{equation*}
where $R(\overset{\rightarrow}{r})=1$ for $r_{1},\ldots ,r_{k}$ satisfying
the equation (**), and $R(\overset{\rightarrow}{r})=0$ in the opposite situation. Since $D$ is $L$-independent, the identity $R(\overset{\rightarrow}{r})=1$ holds only for those $r_{1},\ldots ,r_{k} $ satisfying condition $\sum_{i=1}^{k}\left\Vert r_{i}\right\Vert
_{p}>L$, or $r_{1}=r_{2}=\ldots =r_{k}=0$. Hence,
\begin{equation*}
p^{k-1}I-|C|^{2k}=\sum_{\overset{\rightarrow}{r}\in \mathbb{Z}_{p}^{k},\sum
\left\Vert r_{i}\right\Vert _{p}>L}\hat{B}(r_{1})\cdots \hat{B}%
(\overset{\rightarrow}{r}),
\end{equation*}
as for $r_{i}=0$ the Fourier coefficient $\hat{B}(r_{i})$ is equal to square of the
size of $C$. So, by showing that 

\begin{equation*}
\left\vert C\right\vert ^{2k}>\sum_{\sum \left\Vert r_{i}\right\Vert
_{p}>L}\left\vert \hat{B}(r_{1})\right\vert \cdots \left\vert \hat{B}%
(r_{k})\right\vert R(\overset{\rightarrow}{r}),
\end{equation*}%
we will confirm that $I\neq 0$.

The property of $L$-independence of the set $D$ implies that in any
nontrivial solution of (**) there is some $r_{i}$ satisfying $\left\Vert
r_{i}\right\Vert _{p}>\frac{L}{k}$. The estimates for those $r_{i}$  
\begin{equation*}
\left\vert \hat{B}(r_{i})\right\vert=
\left\vert \hat{C}(r_{i})\right\vert^2
 \leqslant \left (\frac{p}{2r_{i}}\right )^2
 \leqslant \left (\frac{kp}{2L}\right )^2
\end{equation*}%
results from Lemma \ref{wsp Fouriera}.

Denote by $\overset{\rightarrow}{r_j}=(r_1,\cdots,r_{j-1},r_{j+1},\cdot,r_k)$, the vector $\overset{\rightarrow}{r}$ with $j^{th}$ coordinate missing . Substituting this to the previous sum we obtain
\begin{align*}
&\sum_{\sum \left\Vert r_{i}\right\Vert _{p}>L}\left\vert \hat{B}%
(r_{1})\right\vert \cdots \left\vert \hat{B}(r_{k})\right\vert
R(r_{1},\ldots ,r_{k}) \\
\leqslant&\Big ( \frac{kp}{2L}\Big )^2\sum\limits_{j=1}^{k}\sum_{\overset{\rightarrow}{r_j}\in \mathbb{Z}_{p}^{k-1}}\left\vert \hat{B}%
(r_{1})\right\vert \cdots \left\vert \hat{B}(r_{j-1})\right\vert \left\vert 
\hat{B}(r_{j+1})\right\vert \ldots \left\vert \hat{B}(r_{k})\right\vert  \\
\leqslant &k\Big ( \frac{kp}{2L}\Big )^2\sum_{\overset{\rightarrow}{r_k}\in \mathbb{Z}%
_{p}^{k-1}}\left\vert \hat{B}(r_{1})\right\vert \cdots \left\vert \hat{B}%
(r_{k-1})\right\vert.
\end{align*}

The last sum may be estimated further. Let $S_p=\{0,1,\cdots,\frac{p-1}{2}\}$  and we get 
\begin{align*}
&\sum_{\overset{\rightarrow}{r_k}\in \mathbb{Z}_{p}^{k-1}}\left\vert \hat{B}%
(r_{1})\right\vert \cdots \left\vert \hat{B}(r_{k-1})\right\vert  \\
\leqslant &2^{k-1}\sum_{\overset{\rightarrow}{r_k}\in S_p^{k-1}}\left\vert \hat{B}(r_{1})\right\vert \cdots \left\vert \hat{%
B}(r_{k-1})\right\vert.
\end{align*}%

Thus, applying Lemma \ref{wsp Fouriera} again we get%
\begin{align*}
&\sum_{\sum \left\Vert r_{i}\right\Vert _{p}>L}\left\vert \hat{B}%
(r_{1})\right\vert \cdots \left\vert \hat{B}(r_{k})\right\vert
R(\overset{\rightarrow}{r}) \\
\leqslant &k\big(\frac{kp}{2L}\Big)^2
\cdot 2^{k-1}\cdot \Big(\frac{p^{k-1}}{2^{k-1}}\Big)^2\cdot 
\Big (1+\sum_{r\in S_p}\frac{1}{r^2}\Big )^{k-1} \\
\leqslant &k\big(\frac{kp}{2L}\Big)^2
\cdot 2^{k-1}\cdot \Big(\frac{p^{k-1}}{2^{k-1}}\Big)^2\cdot 
\Big (1+\frac{\pi^2}{2}\Big )^{k-1}
\end{align*}%
since $1+\frac{\pi^2}{2}\leq3$, we obtain  
\begin{equation*}
\sum_{\sum \left\Vert r_{i}\right\Vert  {p}>L}\left\vert \hat{B}%
(r_{1})\right\vert \cdots \left\vert \hat{B}(r_{k})\right\vert
R(\overset{\rightarrow}{r})
\leqslant \frac{k^3p^{2k}3^{k-1}}{2^{k+1}L^2}.
\end{equation*}%
So, by the assumption on $L$ we obtain%
\begin{equation*}
\sum_{\sum \left\Vert r_{i}\right\Vert _{p}>L}\left\vert \hat{B}%
(r_{1})\right\vert \cdots \left\vert \hat{B}(r_{k})\right\vert
R(\overset{\rightarrow}{r})< (\varepsilon p )^{2k}\leqslant \left\vert C\right\vert
^{2k}.
\end{equation*}%
This completes the proof.
\end{proof}

\begin{proof}[Proof of Theorem 1]
Let $L$ be a number satisfying inequalities%
\begin{equation*}
 \sqrt{\frac{k^3 3^{k-1}}{2^{k+1}\varepsilon ^{2k}}}<L<\sqrt[k+1]{p}.
\end{equation*}%
Such numbers $L$ exist provided that $p$ is sufficiently large. By \mbox{Theorem \ref {tw L-niezal}}, $\kappa (D)\geqslant \frac{1}{2}-\varepsilon $ for every $L$%
-indepedent set $D$. We show that the second inequality implies that almost
every set in $\mathbb{Z}_{p}^{\ast }$ is $L$-independent. Indeed, the number
of sets that are not $L$-independent is at most%
\begin{equation*}
(2L+1)^{k}\binom{{p-1}}{{k-1}}.
\end{equation*}%
So, the fraction of those sets in $\mathbb{Z}_{p}^{\ast }$ is equal%
\begin{equation*}
\frac{(2L+1)^{k}\binom{{p-1}}{{k-1}}}{\binom{{p-1}}{{k}}}=\frac{%
(2L+1)^{k}k}{p-k}<\frac{(2\sqrt[k+1]{p}+1)^{k}k}{p-k}.
\end{equation*}%
The last expression tends to zero with $p$ tending to infinity. This
completes the proof, as the ratios of two consecutive primes tend to one.
\end{proof}

\section{Integer distance graphs}

We conclude the paper with a remark concerning \emph{integer distance graphs}%
. For a given set $D$, consider a graph $G(D)$ whose vertices are positive
integers, with two vertices $a$ and $b$ joined by an edge if and only if $%
\left\vert a-b\right\vert \in D$. Let $\chi (D)$ denote the chromatic number
of this graph. It is not hard to see that $\chi (D)\leqslant \left\vert
D\right\vert +1$.

To see a connection to parameter $\kappa (D)$, put $N=\left\lceil \kappa
(D)^{-1}\right\rceil $ and split the circle into $N$ intervals $%
I_{i}=[(i-1)/N,i/N)$, $i=1,2,\ldots ,N$ (cf. \cite{RuszaTuzaVoigt}). Let $t$ be a real number such that $min_{d\in D}\| dt\|=\kappa(D)$. Then
define a colouring $c:\mathbb{N}\rightarrow \{1,2,\ldots ,N\}$ by $c(a)=i$ if
and only if $\{ta\}\in I_{i}$. If $c(a)=c(b)$ then $\{ta\}$ and $\{tb\}$ are
in the same interval $I_{i}$. Hence $\left\Vert ta-tb\right\Vert
<1/N\leqslant \kappa (D)$, and therefore $\left\vert a-b\right\vert $ is not
in $D$. This means that $c$ is a proper colouring of a graph $G(D)$. So, we
have a relation%
\begin{equation*}
\chi (D)\leqslant \left\lceil \frac{1}{\kappa (D)}\right\rceil .
\end{equation*}%
Now, by Theorem 1 we get that $\chi (D)\leqslant 3$ for almost every graph $%
G(D)$.

A different proof of a stronger version of this result has been
recently found by Alon \cite{alon}. He also extended the theorem for arbitrary
Abelian groups, and posed many intriguing questions for general
groups.

\begin{acknowledgement}
I would like to thank Tomasz Schoen for an inspiring idea of using
independent sets, and to Jarek Grytczuk for stimulating discussions
and help in preparation of the manuscript. I thank the anonymous
referees for valuable suggestions concerning the merit of the paper. I
also acknowledge a support from Polish Ministry of Science and Higher
Education (MNiSW) (N N201 271335).
\end{acknowledgement}

\end{document}